\documentclass{baustms}

\citesort

\theoremstyle{bams}
\newtheorem{thm}{Theorem}[section]

\newtheorem{cor}[thm]{Corollary}

\theoremstyle{bamsdefn}

\newcommand{\dia}{$\diamondsuit$ }
\newcommand{\diaa}{$\diamondsuit\!$}

\newcommand{\diabb}{$\blacklozenge\!$}

\newcommand{\whbox}{$\square$ }
\newcommand{\blbox}{$\blacksquare$ }

\newcommand{\blboxx}{$\blacksquare$\hskip -0.1 pt}

\begin{document}
\runningtitle{Short title for running head (top of right hand page)}
\title{PERFECT COLOURINGS OF ISONEMAL FABRICS BY THIN STRIPING}
\cauthor 
\author[1]{R.S.D.~Thomas}
\address[1]{St John's College and Department of Mathematics, University of Manitoba,
Winnipeg, Manitoba  R3T 2N2  Canada.\email{thomas@cc.umanitoba.ca}}

\authorheadline{R.S.D.~Thomas}



\begin{abstract}
Perfect colouring of isonemal fabrics by thin striping of warp and weft and the closely related topic of isonemal prefabrics that fall apart are reconsidered and their relation further explored. The catalogue of isonemal prefabrics that fall apart is extended to order 20 for those of even genus.
\end{abstract}


\maketitle

\section{Introduction}
\noindent Except for a finite list of interesting exceptions, Richard Roth \cite{R1} has classified isonemal periodic prefabric designs into 39 infinite species falling into three more general classes as well as the previously defined genera \cite{C1}. 
Species 1--10 have reflection or glide-reflection symmetries with parallel axes and no rotational symmetry, not even half-turns. 
Species 11--32 have reflection or glide-reflection symmetries with perpendicular axes, hence half-turns, but no quarter-turns. 
Species 33--39 have quarter-turn symmetries but no mirror or glide-reflection symmetries. 
This taxonomy has been refined slightly and used in \cite{P1,P2,P3}, to which reference needs to be made, to determine the feasible symmetry groups and hence isonemal prefabrics.
As Roth observes beginning his subsequent paper \cite{R2} on perfect colourings, `[r]ecent mathematical work on the theory of woven fabrics' begins with \cite{ST}, which remains the fundamental reference.
In that paper Roth determines which fabrics---actually prefabrics---can be perfectly coloured by striping warp and weft.
This paper is intended to reconsider that topic in terms of Roth's taxonomy as refined in \cite{P1,P2,P3} and also to consider further the related question which of isonemal prefabrics fall apart.

A {\it prefabric,} as defined by Gr\"unbaum and Shephard \cite{IF}, consists of two or more congruent layers (here only two) of parallel strands in the same plane $E$ together with a preferential ranking or ordering of the layers at every point of $E$ that does not lie on the boundary of a strand.
The points not on the boundary of a strand are naturally arranged into what are called here {\it cells}, in each of which one strand is uppermost.
The (parallel) strands of each layer are perpendicular to those of the other layer, making the cells square.
These square cells are taken here to be of unit area.
The mathematical literature on weaving has concerned exclusively {\it periodic} arrangements in the plane in the standard two-dimensional sense explained by Schattschneider \cite{S1}.
There exists a non-unique finite region and two linearly independent translations such that the set of all images of the region, when acted upon by the group generated by these translations, reproduces the original configuration, which is assumed to be infinite in all directions for convenience.
Schattschneider gives the name {\it unit} to a smallest region of the plane having the property that the set of its images under this translation group covers the plane.
Such units are all of the same area, the {\it period}, but in general can be of a variety of shapes. 
Since our prefabric layers meet at right angles and the symmetry groups with which we shall be concerned here are all rotational ($p4$) or have axes of reflection or glide-reflection in only parallel or in two perpendicular directions, the {\it lattice units}, that is, period parallelograms whose vertices are all images of a single point under the action of the translation subgroup, can be either rectangular or rhombic (Figure 1).
\begin{figure}
\centering
\includegraphics{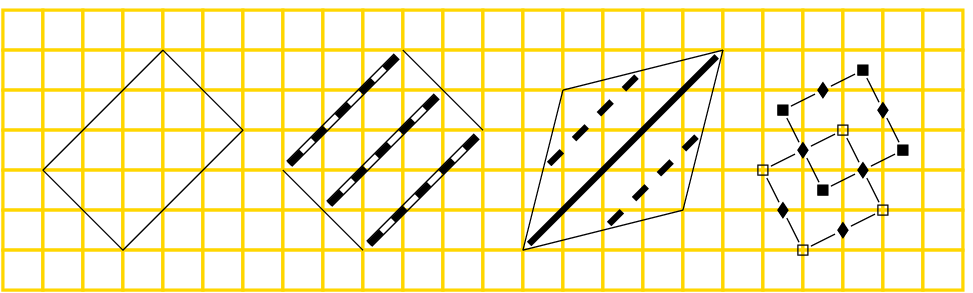}

(a) \hskip 0.75 in (b) \hskip 0.75 in (c) \hskip 0.75 in (d)
\vskip 5pt
\includegraphics{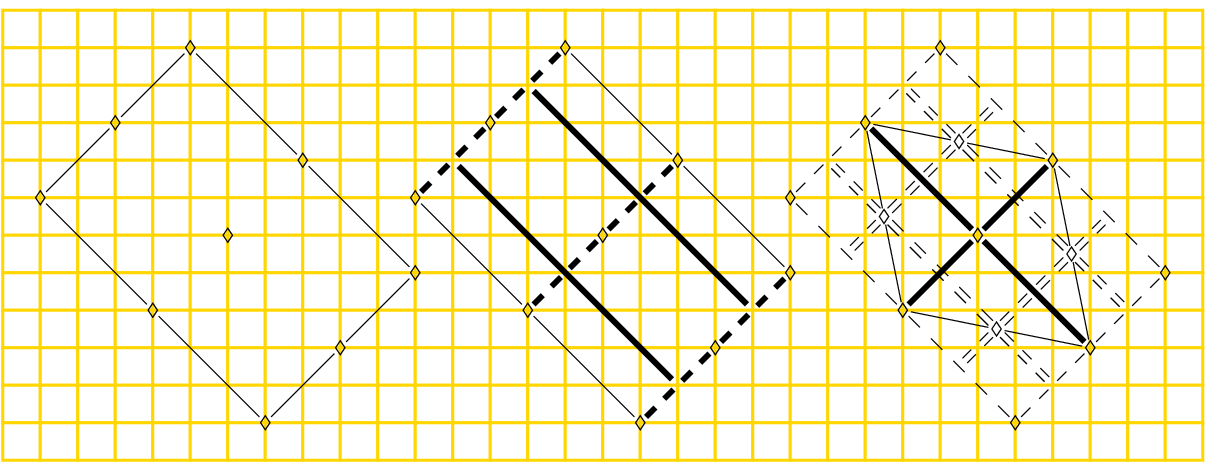}

(e)\hskip 1.5 in (f)\hskip 1.5 in (g)
\caption{The symbols for one lattice unit of some symmetry groups and side-preserving subgroups. \break
a. Side-preserving subgroup of crystallographic type $p1$ with an outline of no geometrical significance but its dimensions. \hskip 10 pt 
b. Symmetry group of type $pm$ with mirrors that become axes of glide-reflection in the side-preserving subgroup. \hskip 10 pt 
c. Type-$cm$ symmetry group with mirrors and axes of side-reversing glide-reflection alternating. \hskip 10 pt 
d. Type-$p4$ symmetry group with alternative lattice units outlined. Centres of quarter-turns without side reversal \whbox and with side reversal \blbox are illustrated. Also centres of side-reversing half-turns \diabb.\hskip 10 pt 
e. Side-preserving subgroup of type $p2$ with an outline of no geometrical significance through centres of half-turns \diaa. \hskip 10 pt 
f. Type-$pmg$ symmetry group with mirrors and side-reversing glide-reflections. \hskip 10 pt 
g. The rhombic lattice unit of a symmetry group of type $cmm$ with four centres of half-turns outside it marking the corners of the dashed outline of a lattice unit of the side-preserving subgroup of type $pgg$.}\label{fig1:}
\end{figure}
Some of the rectangles have one set of parallel boundaries defined by the group but the perpendicular boundaries arbitrary in position, only the distance between them being dictated by the group (e.g., Figures 1b and 6).
The notion of symmetry group allows the definition of the term isonemal; a prefabric is said to be {\it isonemal} if its symmetry group is transitive on the strands, whose directions are conventionally chosen to be vertical, called {\it warps}, and horizontal, called {\it wefts}.
The behaviour of each strand in a prefabric is, by isonemality, the same as the behaviour of every other strand; the periodicity in two dimensions entails periodicity in one dimension along each strand.
To distinguish the linear periodicity from the planar, the period along each strand will be called {\it order}.
Fabrics were catalogued by Gr\"unbaum and Shephard, sorted first under order, and then by binary index, which is the sequence of pale and dark cells of order length represented by 0s and 1s and chosen to be a minimum, then given an arbitrary sequence number.
The first catalogue \cite{C1} listed fabrics other than twills for orders 2 to 8 (plus 13) and an extension \cite{C2} likewise up to order 12 (plus 15 and 17).
Prefabrics that fall apart were catalogued by Hoskins and Thomas \cite{JA} for orders 4 to 16---numerically the same way but with an asterisk to indicate that they are not fabrics.

The standard way to represent the preferential ranking of the strands is to regard the plane $E$ as viewed from one side, from which viewpoint one or the other strand is visible in each cell.
By the {\it normal colouring} of warps dark and wefts pale, the visual appearance of the strands from a particular viewpoint becomes an easily understood code for which strand is uppermost.
This visual appearance from one side (arbitrarily chosen and called the {\it obverse}) of the strands we refer to as the {\it pattern}.
Such a pattern, which consists of an array of dark and pale congruent cells tessellating the plane is given a topological meaning, the {\it design} of the prefabric.
As we shall be considering patterns different from the design of a prefabric, the distinction is important; a design is the pattern of a prefabric that is normally coloured.
So a pattern can be interpreted as a design or not.
Normal colouring has a consequence that other colourings of the strands need not have.
Because at every point not on the boundary of a strand the two strands preferentially ranked are of different colours, a design's colour complement (switching dark and pale) is the appearance of the prefabric from behind as though viewed from the obverse side in a mirror set up behind it.
This we shall call the {\it reverse} of the prefabric.
When a prefabric is coloured normally the reflected reverse pattern is the colour complement of the obverse pattern.

The distinction between prefabrics and fabrics can now be explained; a {\it fabric} is a prefabric that {\it hangs together}, that is, that does not {\it fall apart} in the sense that some warps or some wefts or some of each can be lifted off the remainder because they are not bound into a coherent network by the interleaving defined by the preferential ranking.
The most extreme example of a prefabric that falls apart is the trivial prefabric in which all warps pass over all wefts (uniformly dark diagram) or all wefts pass over all warps (uniformly pale diagram).
In an obvious extension of the notation adopted in the catalogue of genuinely periodic isonemal prefabrics that fall apart \cite{JA}, the trivial prefabric would be denoted 1-0-1*.

For isonemal prefabrics other than his short list of exceptions, all of order less than 5, Roth showed that the symmetry group $G_1$ is a layer group with two-dimensional projection of crystallographic type $pg$, $pm$, $cm$, $pgg$, $pmg$, $pmm$, $cmm$, or $p4$, that is, it has respectively glide-reflection axes or mirrors or both in parallel directions or glide-reflection axes in perpendicular directions, or glide-reflection axes and axes of reflection in directions perpendicular to each other, or just perpendicular axes of reflection, or both (alternating) in both perpendicular directions or just quarter-turns and half-turns.
In order for a glide-reflection or quarter-turn to be a symmetry of the prefabric, it may or may not have to be combined with reversal of the sides of the prefabric, $\tau$, i.e., reflection in the plane $E$. 
Mirror symmetry must always be planar reflection combined with $\tau$ because any cell through which the mirror passes must be its own image in the symmetry but the reflection alone, reversing warp and weft, would reverse its conventional colour. So $\tau$ is needed to restore it.
This means that there is never mirror symmetry in the side-preserving subgroup $H_1$, which really is two-dimensional (no $\tau$). 
It is $H_1$ that determines the two-dimensional period under translation alone.
For these species of prefabric, $H_1$ is of type $p1$, which is generated by translations only, or $p2$, which is a group generated by half-turns only, or $pg$, $pgg$, or $p4$ already characterized.
When $H_1$ is of type $pg$, $pgg$, or $p4$ it may be the same group as $G_1$, or it may be a proper subgroup of $G_1$, or it may be just of the same type as $G_1$.

There are too many possible geometrical configurations of the symmetry groups and side-preserving subgroups just mentioned to illustrate their lattice units here.
Most of them are illustrated in \cite{P1,P2,P3}.
There is are tables in \cite{P3} of which figures in \cite{P1,P2,P3} illustrate which types of group.
Figure 1 illustrates a selection in order to show the conventions for the symmetry-group operations.
Glide-reflection axes without $\tau$ (side-preserving) are hollow dashed lines (making the central rectangle of Figure 1g) and with $\tau$ (side-reversing) are filled dashed lines (Figures 1c and f).
Mirrors (always with $\tau$) are filled double lines (Figures 1c, f, and g), but if the site of a mirror is also that of an axis of side-preserving glide-reflection the filling of the double lines is dashed (Figure 1b).
Axes of glide-reflection can run through the centres and corners of cells ({\it mirror position,} Figures 1b, 10, and 20) or through the sides of cells instead (where a mirror cannot go, Figure 1c).
Centres of half-turns are represented as \dia and centres of half-turns with $\tau$ as \diabb ; centres of quarter-turns as \whbox and centres of quarter-turns with $\tau$ as \blbox (Figure 7a).
Thin lines are just boundaries of lattice units, outlining or completing the outline of rectangles or rhombs.
The alternative lattice units outlined in Figure 1d are related by having the corners of each as centres of the other.
As the diagram illustrates with a group of Roth type 39 ($p4/p4$), the corners of the two sorts of lattice unit can be of different kinds.
Finally, in Figure 1g four centres of half-turns are illustrated outside the $G_1$ lattice unit to complete a lattice unit with dashed outline of a type-$p2$ $H_1$ subgroup inside the symmetry group $G_1$ of type $cmm$.
The subgroup is inside the group, but the lattice unit of the group is inside the lattice unit of the subgroup.
The longer and shorter diagonals of a rhomb will be called its {\it length} and {\it width}.

The subject matter now having been introduced, the plan of the paper can be given. 
Striping, thin and thick, is explained in \S 2, and which isonemal prefabrics can be thinly striped to be perfect colourings is determined in terms of Richard Roth's taxonomy.
In \S 3 it is determined which of these stripings must always be the design of an {\it isonemal} prefabric that falls apart.
The catalogue of designs of isonemal prefabrics that fall apart is extended to order 20 in \S 4.
Then in \S 5 it is shown that not all such designs can be produced by thinly striping an isonemal fabric, correcting an error in \cite{JA}.
Thick striping will be further considered elsewhere.
\section{Striping---Mostly Thin}
\noindent We turn now to the matter of perfect colourings of the strands of a prefabric with two colours, the subject of Roth's later weaving paper \cite{R2}.
An introduction to the colouring topic with emphasis on two colours is \cite{S2}.
Since the normal colouring of strands, which allows the visual pattern of the coloured strands (its design) to represent their topological structure, is a perfect colouring, this idea was more drawn attention to than introduced in \cite{JA}.
A symmetry operation is called a {\it colour symmetry} \cite{R2} `if it permutes the colors consistently'.
All the pale strands must be mapped either to pale strands or to dark strands, and correspondingly the dark strands.
If all of the symmetries of a prefabric with coloured strands are colour symmetries, then the choice of the strand colours is said to be {\it perfect} or {\it symmetric}, where I shall use exclusively the former term.
The interaction of design and pattern is important because the relevant symmetry group is that of the prefabric represented by the pattern that is its design, the group of the design for short, but the permutation (identity or reversal) of colours occurs only in the pattern that represents the colouring of the strands, identical to the design only with normal colouring of the strands.
                  There is potential for confusion.
Fortunately, since there are only two non-normal ways to colour a prefabric that can result in perfect colouring, it is easy to make this discussion concrete.
Warps and wefts can be striped, that is vary pale and dark, either {\it thinly}, that is alternately, or {\it thickly}, that is alternating in pairs: pale, pale, dark, dark, pale, pale.
Adapting a device from \cite{R2}, the colouring of a prefabric can be represented by seeming to extend strands outside the pattern to indicate which strands are pale or dark.
I adopt this convention as long as it is not completely obvious which strands are which.
\begin{figure}
\centering
\includegraphics{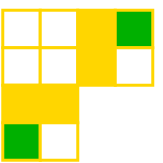}
\hskip 10 pt
\raisebox{22 pt}{\includegraphics{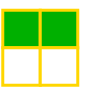}}\hskip 10 pt
\raisebox{22 pt}{\includegraphics{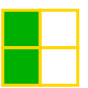}}\hskip 10 pt
\includegraphics{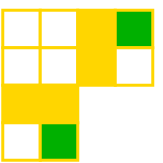}\hskip 10 pt
\raisebox{22 pt}{\includegraphics{2b.eps}}\hskip 10 pt
\raisebox{22 pt}{\includegraphics{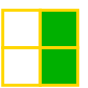}}

\hskip 12pt (a) \hskip 32 pt (b) \hskip 21 pt (c) \hskip 30 pt (d) \hskip 32 pt (e) \hskip 21 pt (f)

\caption{The trivial prefabric 1-0-1*.\hskip 10 pt 
a. Design and specification of one colouring by thin striping. \hskip 10 pt 
b. Front view (obverse) of colouring (a). \hskip 10 pt 
c. Rear reflected (reverse) view of colouring (a).\hskip 10 pt 
d. Specification of second colouring by thin striping. \hskip 10 pt 
e. Obverse view of colouring (d).\hskip 10 pt 
f. Reverse reflected view of colouring (d).}\label{fig2:}
\end{figure}

Consider the thin striping of the trivial prefabric 1-0-1*, whose design is all pale or all dark.
In Figures 2a and 2d, the design of 1-0-1* appears all pale with an indication of how the prefabric is to be coloured, to the left
of the resulting patterns 2b and 2e, where it is seen that the colouring of the invisible warps is not relevant,
and the patterns 2c and 2f, which are the appearances from behind as in a mirror held beyond the the plane of the prefabric.%
\footnote{One reason to look at mirror images of the reverse of the prefabrics is that it is much clearer which cells correspond to cells of the obverse.
Another is to preserve the handedness of patterns with handedness.}
On the reverse side, it is the colouring of the wefts that is irrelevant, but even this trivial example alerts one to the fact that the reverse is not just the colour reversal of the obverse.
Neither can be more important than the other, and so their systematic relation is of interest.
Where the crossing strands are of the same colour, as in the upper-left and lower-right cells of Figure 2a, both obverse and reverse have that colour.
Elsewhere the obverse colour is that of the design in pale rows (short for predominantly pale rows where the weft is pale) and its complement in dark rows (short for predominantly dark rows where the weft is dark).
The thin striping creates a checkerboard of cells that may be called redundant and irredundant, where the {\it redundant} are those where a colour meets itself and the {\it irredundant} are those where the warp and weft have different colours.
The appearance in irredundant cells is dependent on the design and colouring, in redundant cells only on the colouring.
In this language, the irredundant cells on the reverse have the complement of the design in pale rows and the colour of the design in dark rows, in both cases the complement of the obverse.
\begin{figure}
\centering
\includegraphics{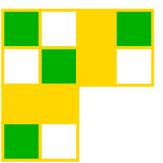}
\hskip 10 pt
\raisebox{22 pt}{\includegraphics{2b.eps}}\hskip 10 pt
\raisebox{22 pt}{\includegraphics{2c.eps}}\hskip 10 pt
\includegraphics{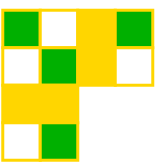}\hskip 10 pt
\raisebox{22 pt}{\includegraphics{2f.eps}}\hskip 10 pt
\raisebox{22 pt}{\includegraphics{2b.eps}}

\hskip 12pt (a) \hskip 32 pt (b) \hskip 21 pt (c) \hskip 30 pt (d) \hskip 32 pt (e) \hskip 21 pt (f)
\caption{Plain weave.\hskip 10 pt 
a. Specification of one colouring by thin striping. \hskip 10 pt b. Obverse view of colouring (a). \hskip 10 pt 
c. Reverse view of colouring (a).\hskip 10 pt 
d. Specification of second colouring by thin striping. \hskip 10 pt 
e. Obverse view of colouring (d).\hskip 10 pt 
f. Reverse view of colouring (d).}\label{fig3:}
\end{figure}
Consider the thin striping of the warp and weft of plain weave.
Because of the interlacing, the different colourings have different results in Figures 3b, c, e, f, where again the reverse is shown on the right.
But results are the same three as those for the trivial prefabric in Figure 2, indicating general non-uniqueness of the results of striping strands.

This striping illustrates that the patterns obtained by striping warp and weft of isonemal fabrics are the designs of prefabrics that fall apart \cite[Lemma 3]{JA}.
In consequence of this fact together with the known conditions for isomenal fabrics that fall apart \cite{CRJC,WD}, if the pattern of a fabric obtained by striping warp and weft is the design of an isonemal prefabric, then the prefabric is of genus II, IV, or V with no overlap with genus I or III and with one quarter of the cells in half the rows dark.
The pattern arising from the striping of warp and weft of an isonemal fabric does not need to {\it be} the design of an isonemal prefabric as we shall see.

The effect of striping warp and weft thickly is to produce a checkerboard of redundant cells that is plain weave doubled (box weave 4-3-1).

If there is to be any hope of perfect colouring, then the colouring, stri\-ping, must be chosen so that the colour symmetries of the fabrics map redundant cells to redundant cells and irredundant cells to irredundant cells, or as Roth puts it `preserve' them.
Which half of the cells are to be reduntant and which half are to be irredundant is a choice to be made.
There are therefore two ways to stripe the same fabric thinly.
Preserving the two classes of cell rules out as symmetries for any striping, 

\noindent glide-reflections with axes not in mirror position, 

\noindent translations $(x, y)$ with $x$ and $y$ not integers of the same parity, and 

\noindent half-turns with centres not at the centre or corner of a cell 

\noindent but not half-turns with centres in those two positions, not other translations and not mirror symmetries.
It rules out, for thin striping, quarter-turns with centres not at the centre of a cell.
Roth has shown \cite{R2} that these modest necessary conditions are also sufficient to allow the two sorts of striping, but not in terms of his symmetry-group types.

When we turn to species without quarter-turns, the glide-reflection constraint alone eliminates species $1_e$, $1_o$, $2_e$, $2_o$, 4, $8_o$, 10, 12, 14, 16, $18_o$, $18_e$, 20, 24, $28_o$, and 32 from being perfectly coloured by striped strands.
Perfect colouring by thin striping is possible for the remainder of the species with only parallel axes of symmetry: $1_m$, $2_m$, 3, $5_o$, $5_e$, 6, $7_o$, $7_e$, $8_e$, and 9.
As Roth points out, types with $G_1$ of type $pgg$ can have strands thinly striped to produce perfect colouring provided that twice each glide is an even multiple of the cell diagonal $\delta$. 
This is always the case when its axis is in mirror position.
As observed in \cite{P2}, the position of glide-reflection axes and the size of glides are related in such a way that it is unsurprising that the glide-reflections eliminated on the basis of position, those in species 12, 14, and 16, also have glides fractional in $\delta$.
Accordingly, species 11, 13, and 15 are always stripable.
Roth also points out that designs with $G_1$ of type $pmg$ allow thin striping with the same constraint on glides, since reflection is always a colour symmetry.
This permits thin striping for species 17, $18_s$ (s for stripable thinly), 19, 21, 22, and 23, but not $18_o$, $18_e$, 20, and 24, already banned.
A similar blanket permission is for $G_1$ of type $pmm$, that is types 25 and 26, since they have only reflections.
When $G_1$ is of type $cmm$, the stripability depends on the translations generated by the mirrors and glide-reflections, which are correctly located except in species 32, already banned.
The translations are the diagonals of the central rectangle as in Figure 1f, and their components must have the same parity. 
Both the odd-odd and even-even spacings of Roth type 27, which are used also in species 29 and 30, and the odd-even spacing of species $28_e$ and $28_n$, the former used again in species 31, allow the use of striping.
So we have seen what can be stated as a theorem, essentially proved in \cite{R2}.

\begin{thm}
Isonemal periodic prefabrics of order greater than $4$ and of every species with symmetry axes that is not ruled out by the placement of glide-reflections can be perfectly coloured by thin striping: $1_m, 2_m, 3, 5_o, 5_e, 6, 7_o, 7_e, 8_e$, $9, 11, 13, 15, 17, 18_s, 19, 21$--$23, 25$--$27, 28_e, 28_n$, and $29$--$31.$
\end{thm}

\begin{thm}
Isonemal periodic prefabrics of order greater than $4$ with quarter-turn symmetry can be perfectly coloured by thin striping if and only if they are of species $36_s$.
\end{thm}

\begin{proof} The type specifications $36_s$ correspond exactly to the conditions for thin striping of a prefabric with $p4$ symmetry. (This theorem is a second reason for the choice of subscript.)
\end{proof}

\section{Producing isonemal prefabrics that fall apart}
\noindent It is a curiosity that the only perfect colourings illustrated in \cite{R2} are normally coloured designs.
It is surely of interest to see what non-normal perfect colourings look like.
To satisfy this curiosity, fabrics of each of the species with only parallel axes of symmetry will be described or displayed perfectly coloured by thin striping.

The fabrics shown in Figures 12a, 14c, and 2a of \cite{P1} to illustrate species $2_m$, $5_o$, and $5_e$ respectively become non-isonemal designs both ways they are striped thinly.
The two ways are illustrated for each example in Figure 4a to f.
\begin{figure}
\centering
\includegraphics{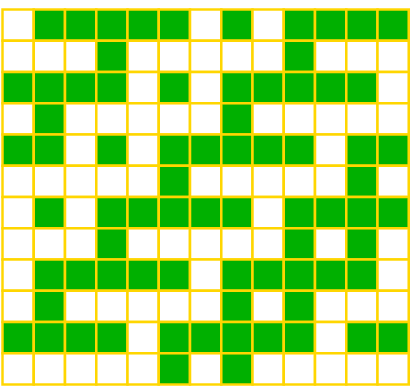}\hskip 10 pt
\includegraphics{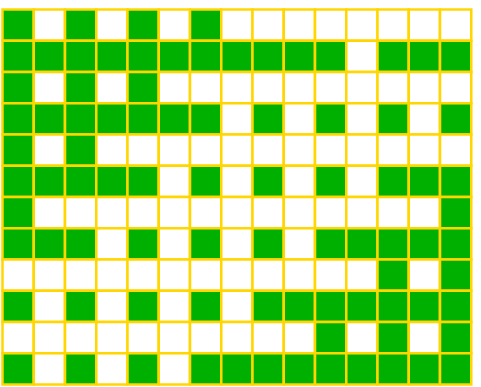}

\noindent (a)\hskip 117 pt (b)

\vspace {4 pt}
\noindent
\includegraphics{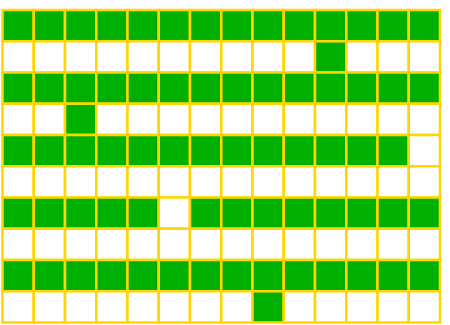}\hskip 5 pt
\includegraphics{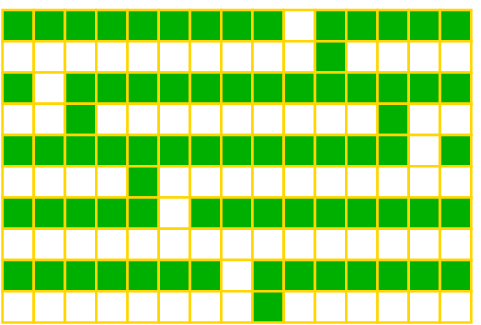}

\noindent (c)\hskip 122 pt (d)

\vspace {4 pt}
\includegraphics{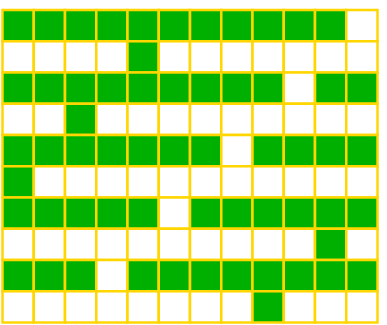}\hskip 10 pt
\includegraphics{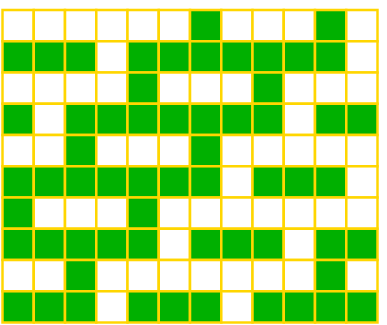}

\noindent (e)\hskip 109 pt (f)
\caption{Patterns (non-isonemal designs) resulting from both ways of colouring with thin striping.
a, b. Example of species $2_m$ in Figure 12a of \cite{P1}. 
\hskip 10 pt c, d. Example of species $5_o$ in Figure 14c of \cite{P1}. 
\hskip 10 pt e, f. Example (12-35-1) of species $5_e$ in Figure 2a of \cite{P1}.}\label{fig4:}
\end{figure}

Fabric 12-79-1 (Figure 8a of \cite{P1}), showing species 3, becomes 12-65-4* when striped one way and a non-isonemal design the other way (Figure 5a). 
Fabric 8-11-1 (Figure 6a of \cite{P1}), showing species 6, becomes 8-5-3* when striped one way and a non-isonemal design the other way (Figure 5b). 
Fabric 8-19-5 (Figure 17a of \cite{P1}), showing species $8_e$, becomes 4-1-1* when striped one way and a non-isonemal design the other way (Figure 5c). 
Fabric 8-11-2 (Figure 7a of \cite{P1}), showing species 9, becomes 8-5-1* when striped one way and a non-isonemal design the other way (Figure 5d).
\begin{figure}
\centering
\includegraphics{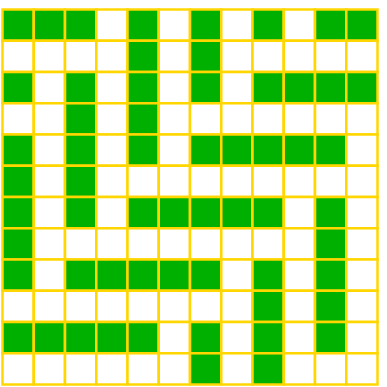}\hskip 10 pt
\includegraphics{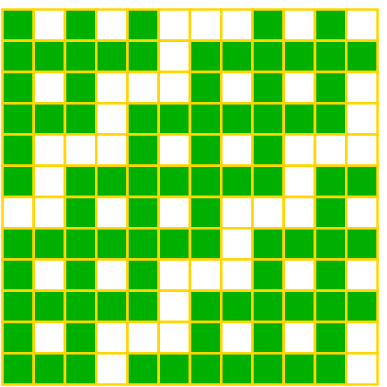}

(a)\hskip 108 pt (b)
\vskip 5pt
\includegraphics{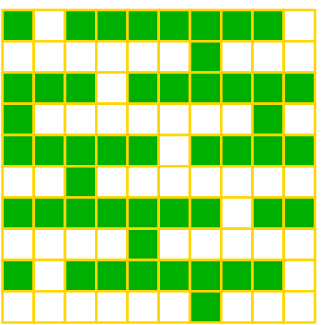}\hskip 10 pt
\includegraphics{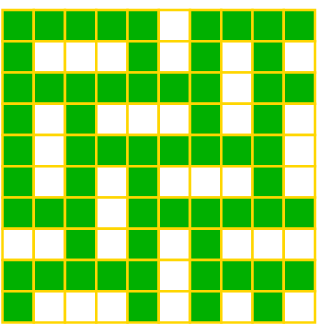}

(c)\hskip 90 pt (d)
\caption{Patterns (non-isonemal designs) resulting from colouring with thin striping.
a. 12-79-1 (species 3). 
\hskip 10 pt b. 8-11-1 (species 6).
\hskip 10 pt c. 8-19-5 (species $8_e$). 
\hskip 10 pt d. 8-11-2 (species 9).}\label{fig5:}
\end{figure}

Fabric 12-183-1 in Figure 4a of \cite{P1}, showing species $1_m$, becomes 12-69-2* or 12-21-2* when thinly striped.
The fabrics in Figures 15a and 16a of \cite{P1}, showing species $7_o$ and $7_e$ respectively (Figures 10a, b, here) both become designs of isonemal prefabrics that fall apart when thinly striped either way (Figures 6a and b, c and d, respectively).
\begin{figure}
\centering
\includegraphics{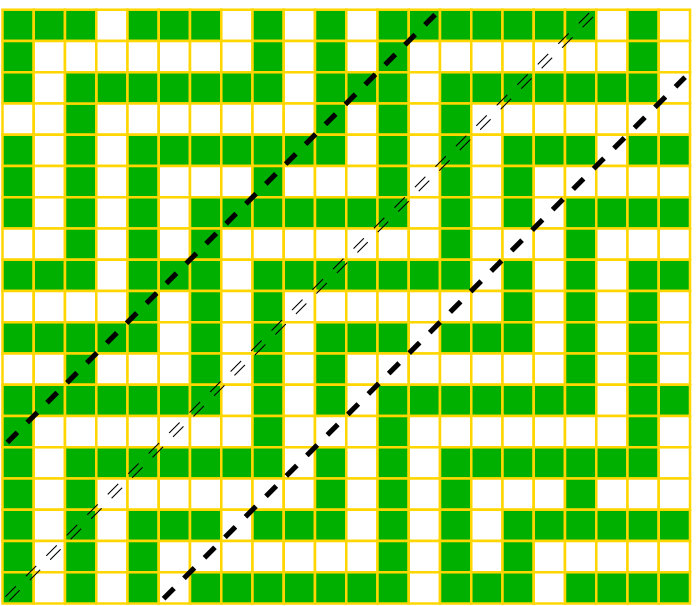}

(a)
\vskip 5pt
\includegraphics{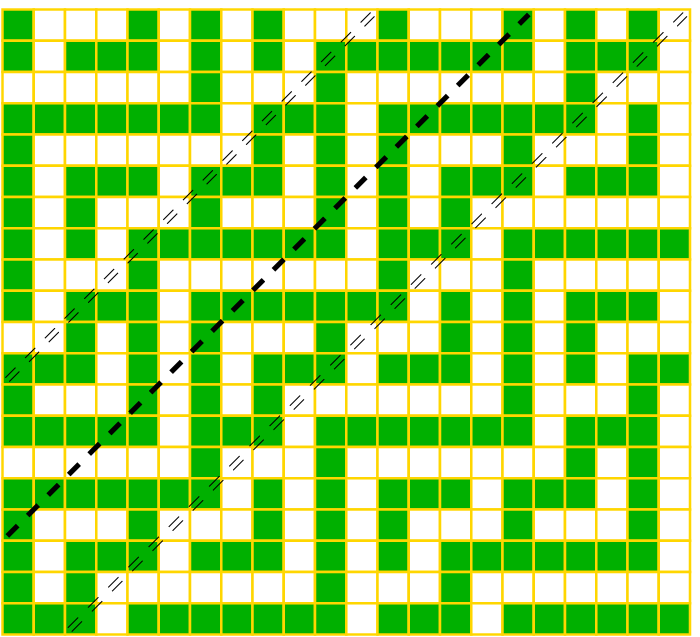}

(b)
\vskip 5pt
\includegraphics{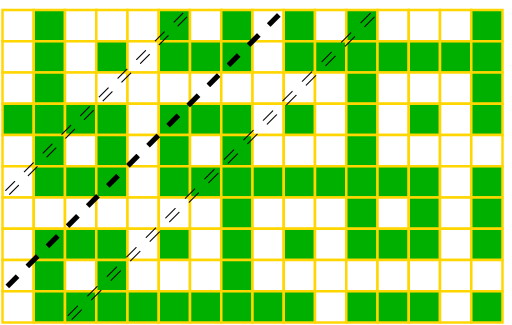}\hskip 5 pt
\includegraphics{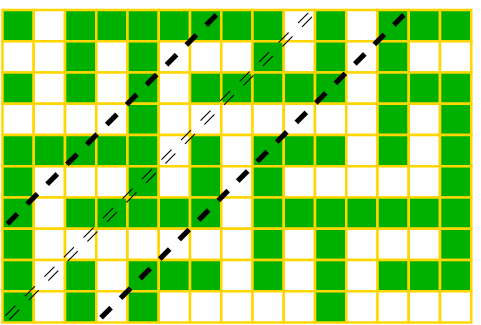}

(c)\hskip 138 pt (d)
\caption{Patterns (species-3 isonemal designs) resulting from two ways of colouring fabrics of species 7 with thin striping.
a, b. Fabric of species $7_o$ in Figure 10a.\hskip 10 pt 
c, d. Fabric of species $7_e$ in Figure 10b.}\label{fig6:}
\end{figure}

What is perhaps most obvious in these examples, in particular Figures 4 and 5, is that there is no tendency to the production of designs of {\it isonemal} prefabrics that fall apart.
More generally, the strand striping rather often produces stripy patterns that are unattractive given isonemal expectations.
This is even true for fabrics with $p4$ symmetry, of which only one species can be perfectly coloured by thin striping.
The example that Roth mentions of a design of species $36_s$ is 10-1-1, the $(10, 3)$ satin. 
It is illustrated, with its two thin stripings, in Figures 7 and 8.
\begin{figure}
\centering
\noindent
\includegraphics{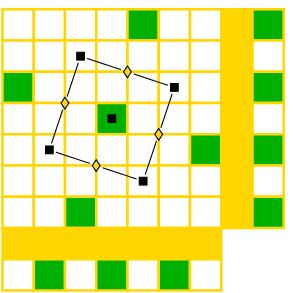}\hskip 10 pt
\raisebox{18 pt}{\includegraphics{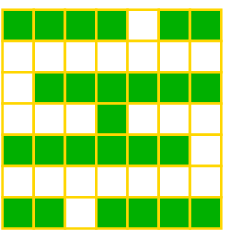}\hskip 10 pt\includegraphics{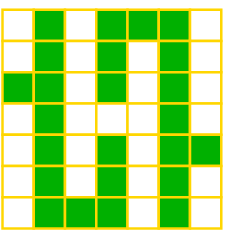}}

\noindent (a)\hskip 82 pt (b)\hskip 60 pt (c)

\caption{Specification of one colouring by thin striping of a fabric of species $36_s$, 10-1-1, the $(10, 3)$ satin with symmetries noted.\hskip 10 pt a. Specification.\hskip 10 pt
b. Obverse.\hskip 10 pt 
c. Reverse.}\label{fig7:}
\end{figure}
\begin{figure}
\centering
\noindent
\includegraphics{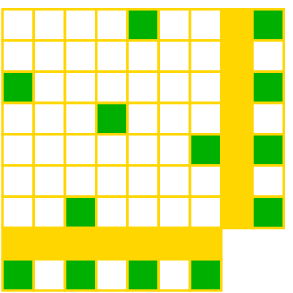}\hskip 10 pt
\raisebox{18 pt}{\includegraphics{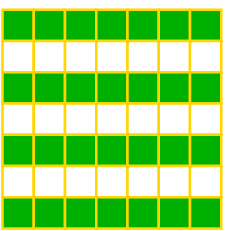}\hskip 10 pt\includegraphics{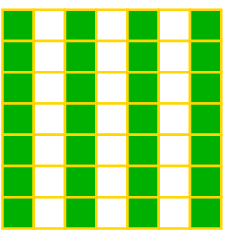}}

\noindent (a)\hskip 82 pt (b)\hskip 60 pt (c)
\caption{a. Specification of the other colouring of the fabric of Figure 7a by thin striping.\hskip 10 pt
b. Obverse.
\hskip 10 pt c. Reverse.}\label{fig8:}
\end{figure}
One notes that Figures 7 and 8 emphasize that \blbox means rotate {\it and reflect;} without the $\tau$ the pattern is changed quite seriously by a quarter turn, as those of Figures 4 and 5 are by their symmetries, which include $\tau$, if $\tau$ is forgotten.

The predominance of stripiness persists in species 11--32, where a similar analysis of examples illustrated in \cite{P2} gives the following results.
The example of species $18_s$ in Figure 9b of \cite{P2} (12-135-1),
of species 26 in Figure 4b of \cite{P2} (12-619-1),
of species 27 in Figure 12a of \cite{P2} (6-1-1),
of species $28_e$ in Figure 13a of \cite{P2} (8-1-1, the $(8, 3)$ satin),
of species 29 in Figure 15a of \cite{P2} (16-2499 of \cite{R1}),
and of species 31 in Figure 15b of \cite{P2} (order 24),
when thinly striped, have patterns neither of which is the design of an isonemal prefabric.
This behaviour is analogous to that illustrated in Figure 4.
These are just examples and not chosen for this purpose; it is quite possible for other examples to behave differently.

The largest class of examples are the next.
The example of species 13 in Figure 7a of \cite{P2}, becomes the design of an unknown order-24 isonemal prefabric when striped thinly one way and the design of a non-isonemal prefabric when striped thinly the other way.
The same is true of further examples as follows, where the isonemal result is indicated in parentheses.

\noindent Species 15 in Figure 8a of \cite{P2} (12-47-5 $\rightarrow$ 12-21-2*).

\noindent Species 17 in Figure 9a of \cite{P2} (8-19-7 $\rightarrow$ 4-1-1*).

\noindent Species 19 in Figure 9c  of \cite{P2} (8-19-4 $\rightarrow$ 4-1-1*).

\noindent Species 21 in Figure 10a  of \cite{P2} (8-7-2 $\rightarrow$ 8-5-1*).

\noindent Species 23 in Figure 11b of \cite{P2} (12-31-1 $\rightarrow$ 12-21-1*).

\noindent Species 25 in Figure 2a of \cite{P2} (8-19-6 $\rightarrow$ 4-1-1*).

\noindent Species $28_n$ in Figure 13b of \cite{P2} (8-5-1 $\rightarrow$ 8-5-3*).

Examples of designs of only three species produce designs of isonemal prefabrics when striped each way, those of species
11 in Figure 5a of \cite{P2}, 
species 22 in Figure 11a of \cite{P2}, and 
species 30 in Figure 16 of \cite{P2}. 
They are illustrated in Figure 9.
\begin{figure}
\centering
\noindent
\includegraphics{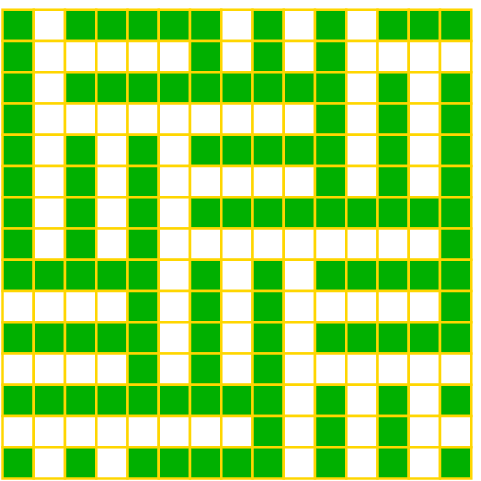}\hskip 10 pt\includegraphics{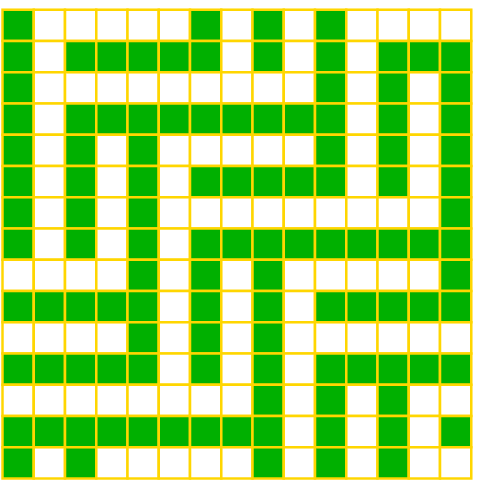}

\noindent (a) \hskip 133 pt (b)

\smallskip\noindent
\includegraphics{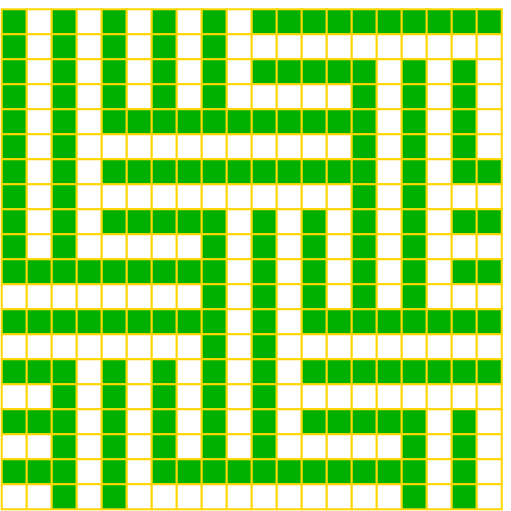}\hskip 5 pt\includegraphics{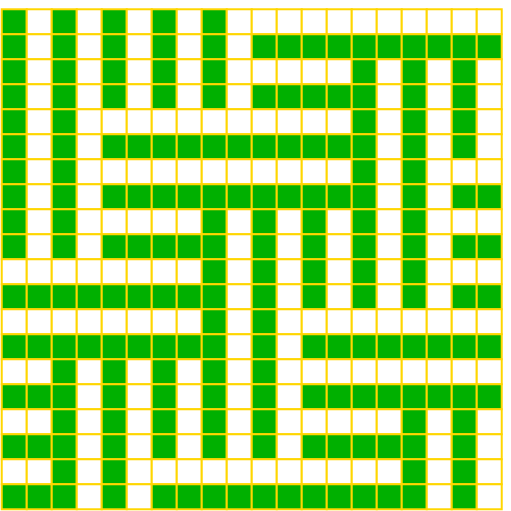}

\noindent (c) \hskip 137 pt (d)

\smallskip\noindent
\includegraphics{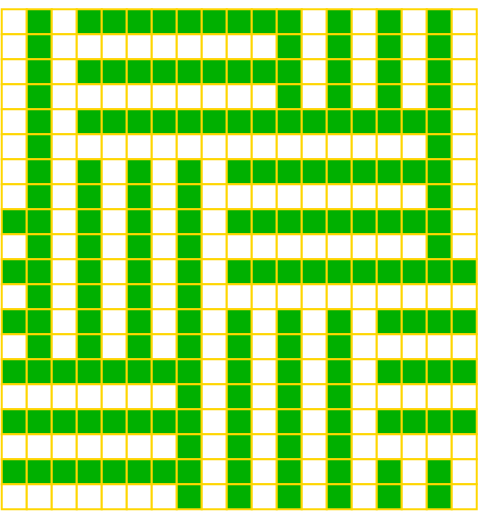}\hskip 10 pt\includegraphics{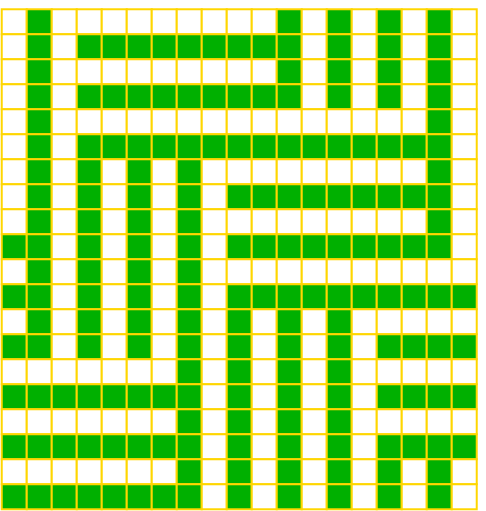}

\noindent (e)\hskip 138 pt (f)
\caption{Isonemal designs resulting from both ways of colouring with thin striping.
a, b. Example of species 11 in Figure 5a of \cite{P2}. 
\hskip 10 pt c, d. Example of species 22 in Figure 11a of \cite{P2}. 
\hskip 10 pt e, f. Example of species 30 in Figure 16 of \cite{P2}.}\label{fig9:}
\end{figure}
These patterns, all of species 15, would be mere curiosities if they did not illustrate a general fact.
And they {\it are} just examples.
(Other examples would give different indications, some of them unhelpful.
Species $28_n$ could have been illustrated by 8-27-5, which becomes 4-1-1* and 8-5-3* when striped both ways.)
But {\it all} fabrics of species 11, 22, and 30 produce designs of isonemal prefabrics when striped both ways, as do those of species $1_m$ and species 7 from the species with only parallel axes.
The examples given above show that {\it only} those species always do so.

\begin{thm}{The pattern of a fabric of order greater than $4$ perfectly coloured by thinly striping strands is the design of an isonemal prefabric that falls apart if the side-preserving subgroup of the fabric's symmetry group is generated by side-preserving glide-reflections and is transitive on strands.}
\end{thm}

The side-preserving subgroups $H_1$ of species $1_m$, 3, 6, $7_o$, $7_e$, and 9 are of type $pg$.
And those of species 11, 15, 19, 21, 22, 23, 29, and 30 are of type $pgg$.
$H_1$ cannot be transitive on strands in these species other than $1_m$, 7, 11, 22, and 30, for fabrics in the other species are of genus II, IV, both II and IV, or V.
So the theorem excludes the species that we know need excluding.

\begin{proof} It is known that the prefabric falls apart; it is required only to prove that it is isonemal.
The side-preserving glide-reflections that generate the side-preserving subgroup of $G_1$, transitive on strands, have even or odd glides, and their axes pass through or between the imposed redundant cells, giving four situations to consider.
They will be considered in pairs.

There are two patterns involved here.
The first, which is the design of the fabric that will be subjected to the striping, and the second, which is to be shown to be the design of an isonemal prefabric.
The first will be called the design, and the second will be called the pattern.

\noindent {\it First pair.} Each even glide-reflection with axis through the redundant cells and odd glide-reflection with axis between the redundant cells preserves the colour of the redundant cells while setting up a correspondence between

predominantly dark columns and

predominantly dark rows in which the irredundant cells have been complemented 

\noindent and between

predominantly pale columns in which the irredundant cells have been complemented and

predominantly pale rows.

\noindent In the design of the fabric on which the pattern (not yet shown to be the sort of design desired) is based, the side-preserving glide-reflections mapped the irredundant cells of (now) predominantly dark columns to irredundant cells of predominantly dark rows, their colour being reversed by the colouring convention for designs.
Their colours did not match.
In the pattern, the irredundant cells of predominantly dark columns are mapped to irredundant cells of predominantly dark rows, the colours of which rows have been reversed by the complementation.
In the pattern there is no colouring convention.
Their colours now do match.
The colours of the corresponding irredundant cells match, and the colours of the redundant cells match.
Exactly the same is true for the predominantly pale columns mapped to predominantly pale rows except that the complementation occurred in the columns instead of the rows.
This consistency means that the pattern has a colour-preserving glide-reflection for each such side-preserving glide-reflection of the fabric design.
Each such colour-preserving glide-reflection is a side-reversing glide-reflection of the pattern when it is regarded as a design.
 
\noindent {\it Second pair.} Each even glide-reflection with axis between the redundant cells and odd glide-reflection with axis through the redundant cells reverses the colour of the redundant cells while setting up a correspondence between

predominantly dark columns and

predominantly pale rows 

\noindent and between

predominantly pale columns in which the irredundant cells have been complemented and

predominantly dark rows in which the irredundant cells have been complemented.

\noindent In the design of the fabric on which the pattern is based, the side-preserving glide-reflections mapped the irredundant cells of predominantly dark columns to irredundant cells of predominantly pale rows, their colour being reversed by the colouring convention for designs.
Their colours did not match.
In the pattern, the irredundant cells of predominantly dark columns are mapped to irredundant cells of predominantly pale rows as before.
Their colours do not match.
The colours of the corresponding irredundant cells are opposite, and the colours of the redundant cells are opposite.
Exactly the same is true for the predominantly pale columns mapped to predominantly dark rows except that both have been complemented, which is immaterial to oppositeness of colour.
This consistency opposite to that of the previous paragraph means that the pattern has a colour-reversing glide-reflection for each such side-preserving glide-reflection of the fabric design.
Each such colour-reversing glide-reflection is a side-preserving glide-reflection of the pattern when it is regarded as a design.

If the pattern is now regarded as a design, it has among its symmetry operations a subgroup generated by side-preserving and/or side-reversing glide-reflections (depending on spacing of axes and placement of the redundant cells) that is transitive on the strands because the glide-reflections from which they came were transitive on the strands.
The symmetry group of the prefabric defined by the design is accordingly transitive on the strands.
And so the prefabric is isonemal as required.
\end{proof}
\begin{figure}
\centering
\noindent
\includegraphics{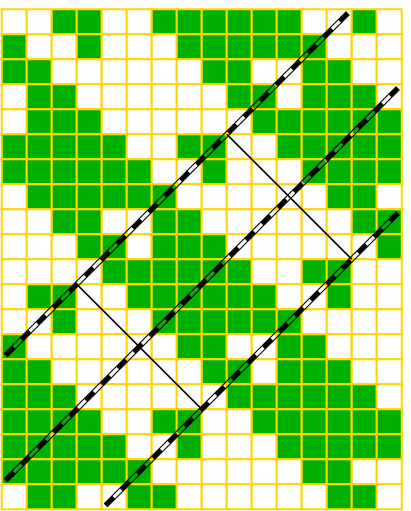}
\hskip 10 pt\includegraphics{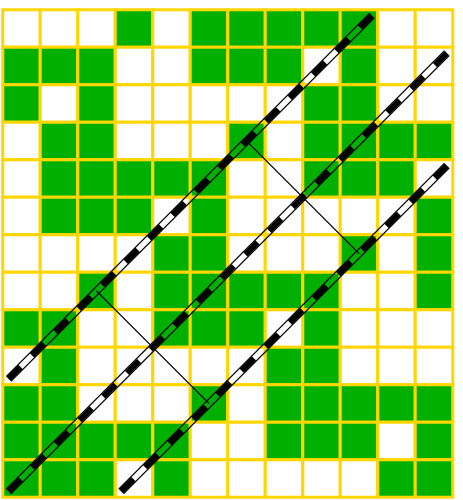}

\noindent (a)\hskip 116 pt (b)
\caption{Illustrations of fabrics of species 7 from \cite{P1}.
a. Figure 15a: subspecies $7_o$. 
\hskip 10 pt b. Figure 16a: subspecies $7_e$.}\label{fig10:}
\end{figure}

It may be helpful to consider the examples available in Figure 6:

\noindent Figure 6a, one colouring of the species-$7_o$ example of Figure 10a drops the bottom row of that figure.
The dark redundant cells begin in the bottom left corner, as does the central axis of side-preserving glide-reflection with glide $3\delta$.
Because it runs through redundant cells, it becomes the axis of side-preserving glide-reflections in the prefabric of Figure 6a.
Because the axes $2.5\delta$ distant run between redundant cells, they become axes of  side-reversing glide-reflections in Figure 6a.

\noindent Figure 6b, the other colouring, drops the first two columns of Figure 10a.
The dark redundant cells begin in the bottom left corner.
The central (side-preserving) glide-reflection axis with odd glide $3\delta$ begins on the left in the cell four up.
Because it runs between redundant cells, it becomes an axis of side-reversing glide-reflection in Figure 6b.
Because the axes $2.5\delta$ distant run through redundant cells, they become axes of side-preserving glide reflection in Figure 6b.

\noindent Figure 6c, one colouring of the species-$7_e$ example of Figure 10b drops the first column of that figure and has pale redundant cells and the central axis of (side-preserving) glide-reflection begin one cell above the bottom left corner.
Since that axis runs through redundant cells with even glide $2\delta$, it becomes the axis of side-reversing glide-reflections marked in Figure 6c.
The other marked axes, $1.5\delta$ distant and running between redundant cells, become axes of side-preserving glide-reflection in Figure 6c.

\noindent Figure 6d, the other colouring, has its bottom left corner match that of Figure 10b.
Dark redundant cells begin one cell up, and the central axis of (side-preserving) glide-reflection with glide $2\delta$ runs between redundant cells.
Accordingly, it becomes the axis of side-preserving glide-reflection in Figure 6d.
The axes $1.5\delta$ distant run through redundant cells, and so they become axes of glide-reflections that are side-reversing.

\begin{cor}{If the side-preserving subgroup of a fabric's symmetry group is generated by side-preserving glide-reflections and is transitive on strands and the fabric is of order greater than $4$ and is perfectly coloured by thinly striping strands, then the pattern
 is the design of a prefabric having a side-preserving glide-reflection where the fabric had a side-preserving glide-reflection either between redundant cells with even glide or through redundant cells with odd glide and having a side-reversing glide-reflection where the fabric had a side-preserving glide-reflection either through redundant cells with even glide or between redundant cells with odd glide.}
\end{cor}

\smallskip
This is exactly what the mechanics of the above proof show.
No similar result is true for side-reversing glide-reflections.
For this reason one might expect that thinly stripable species for which there are no side-preserving glide-reflections, namely $8_e$ (Figure 5c), 13, 17, 25, and $28_n$, in addition to those listed above as failing to produce two isonemal designs, might also fail. 
Indeed, larger examples from species $8_e$, 13, 17, 25, and $28_n$ do fail to give designs of isonemal prefabrics when striped either way, so that it is only those that have side-preserving glide-reflections that have dependably even one isonemal design as a result; they have {\it one} because it is impossible for the axes to avoid the positions covered by the Corollary.
\begin{thm}
The pattern of a fabric of order greater than $4$, perfectly coloured by thinly stiping strands, cannot be the design of an isonemal prefabric that falls apart if the symmetry group of the fabric has quarter-turn symmetry.
\end{thm}

\begin{proof}  The isonemal fabric coloured, being of species $36_s$ by Theorem 2.2, is of level 2.
Because the movement from the corner of a level-2 (in the terms of \cite{P3}) lattice to its centre is an even number of cell widths in one vertical or horizontal direction and then an odd number of cell widths in a perpendicular direction, the lattice-unit corners and centres fall in both the centres of redundant cells and the centres of irredundant cells. 
The \blboxx s in irredundant cells are in predominantly pale rows and dark columns or dark rows and pale columns.
These rows and columns cannot be related by \blbox but by \whbox as in Figures 7b, c (central \blboxx ) and 8b, c (non-central \blboxx s).
But \whbox cannot occur within a cell as a symmetry.
This contradiction shows that the pattern resulting from colouring by thin striping cannot be the design of an isonemal prefabric.
What goes wrong in the example of Figures 7 and 8 always goes wrong.
\end{proof}

Theorem 3.3, together with the examples before Theorem 3.1, show the sufficient condition of Theorem 3.1 to be necessary as well as sufficient for whole species of fabrics. 
Striping individual fabrics beyond those classes may happen to produce isonemal designs.
\section{The catalogue of isonemal prefabrics that fall apart extended}
\noindent What species can be produced by thin stiping? 
To what species do isonemal designs that fall apart belong?
The latter must be of genera pure II, pure IV, or II and IV with no I, III, or V on account of the 3/4 dark, 1/4 pale feature.
These species are 3, 6, and 9 (genus II), 15 and 19 (genus IV), and 23 and 31 (genera II and IV), including none from 33--39.
The last fact corresponds to the impossibility of producing an isonemal design by thinly striping a non-exceptional isonemal fabric with quarter-turn symmetry.
All of the possible species are represented in the catalogue \cite{JA}: 12-69-2* (3), 12-69-1* (6), 16-277-4* (9), 12-21-4* ($15_o$), 16-277-2* ($15_e$), 12-69-3* ($19_o$), 16-85-3* ($19_e$), 12-21-1* ($23_o$), 8-5-1* ($23_e$), 8-5-3* (31).
It is not difficult to determine all of the isonemal designs of prefabrics that fall apart of any species and order by the method indicated in \cite[Section 5]{P1} using the data collected in \cite[Section 8]{P1} and \cite[Section 12]{P2}.

To take the most obvious case, order $20=2ab$ with $a=5$ and $b=2$ for species 3 and 6 and with $a=1$ and $b=10$ for 6 alone, and $20=4ab$ with $a=1$ and $b=5$ for species $15_o$, $19_o$, and $23_o$ and with $a=5$ and $b=1$ for $23_o$.
Nothing allows species 9 or 31 at order 20.
The lattice units for the 9 possibilities are shown in Figure 11.
\begin{figure}
\centering
\includegraphics{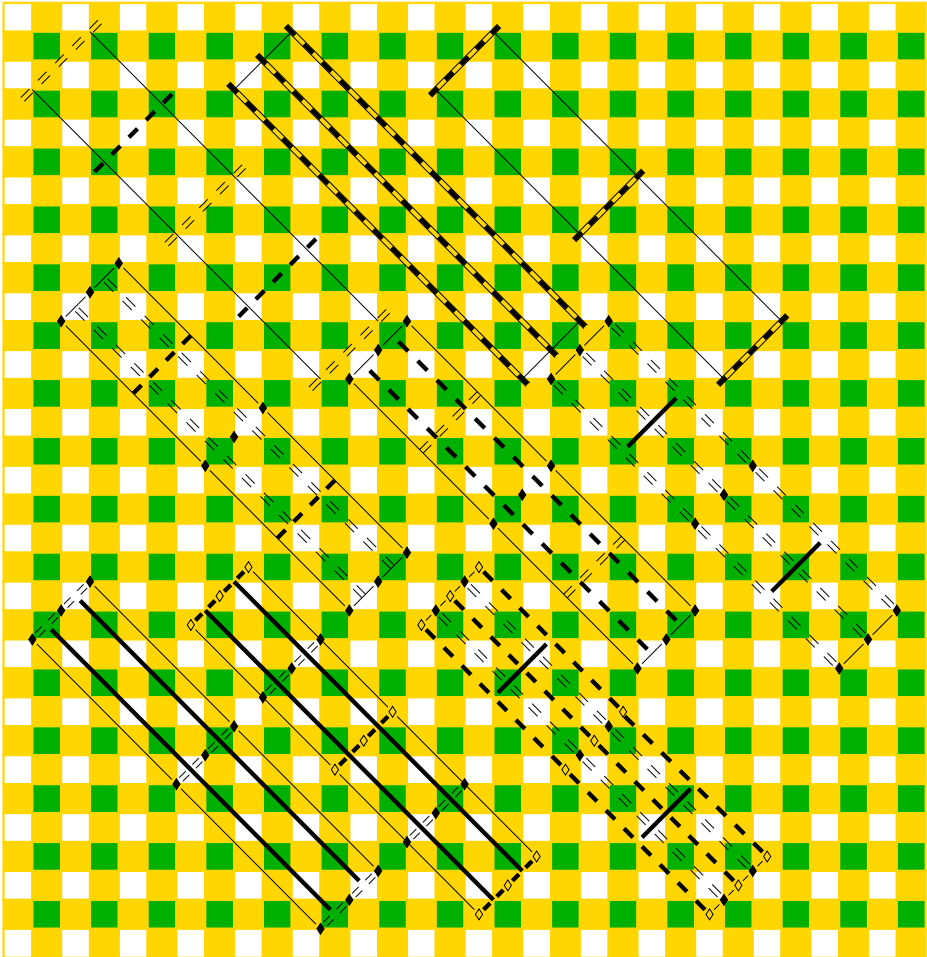}
\caption{The lattice units of the species of prefabrics of order 20 and genus II, IV, or II and IV that fall apart fitted into the matrix of redundant cells, indicated dark and pale rather than neutral. Row 1. a. Species 3.\hskip 10 pt b and c. Species 6.\hskip 10 pt 
Row 2. d and e. Species $15_o$.\hskip 10 pt f. Species $19_o$ one way.\hskip 10 pt Row 3. g. Species $19_o$ the other way.\hskip 10 pt h and i. Species $23_o$.}\label{fig11:}
\end{figure}
One has to fit the axes and centres of the group to the lattice of redundant cells.
For example, for species 3 axes of side-preserving glide-reflection with odd glide would have to lie through redundant cells and the side-reversing glide-reflection axes through irredundant cells, but if the glides are even as in Figure 11a, then the axes of side-preserving glide-reflection must pass instead through {\it irredundant} cells and the axes of side-reversing glide-reflection through the {\it redundant} cells.
That these arrangements are possible is due to the distance in $\delta$ between neighbouring axes, half an odd multiple of $\delta$.
The usual nuisance occurs by which the smaller-order 4-1-1* appears at all nine opportunities because the symmetry groups examined are subgroups of that of 4-1-1*.
And then because of the smallness of one dimension of the lattice units used and because of the extra symmetry imposed by the redundant cells, each set of designs produced contains other designs with more symmetry and therefore of a different species than that intended.
In particular, species 3 (Figure 11a), the second set from species 6 (Figure 11c), the first set from species 15 (Figure 11d), and the second set from species 19 (Figure 11g) all contained the first set of designs of species 23 (Figure 18), and the second set from species 6 (Figure 11c), the second set from species 15 (Figure 11e) and the first set from species 19 (Figure 11f) contained the second set of designs of species 23 (Figure 19).
The first set for species 6 (Figure 11b), in fact, contained nothing but designs that were not of species 6.
\begin{figure}
\centering
\includegraphics{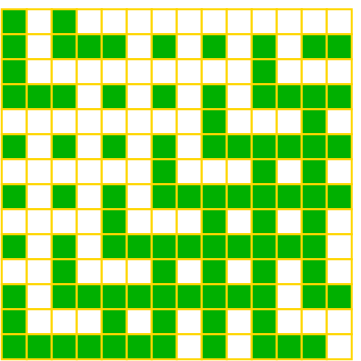}\hskip 5 pt \includegraphics{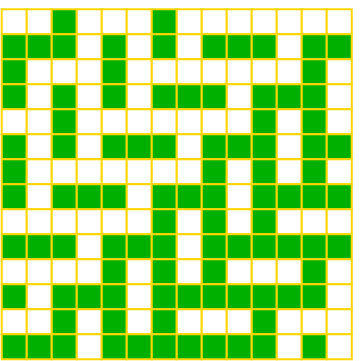}\hskip 5 pt \includegraphics{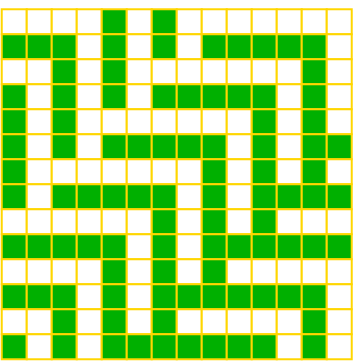}

1109-1* \hskip 65pt 4373-1* \hskip 65pt 5141-1*
\vskip 5pt
\includegraphics{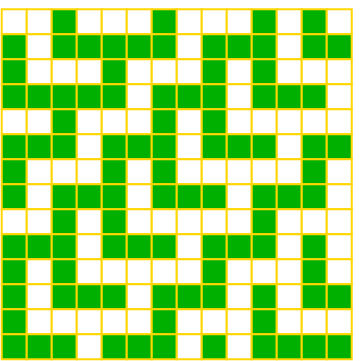}\hskip 5 pt \includegraphics{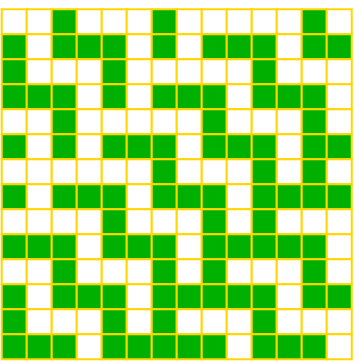}\hskip 5 pt \includegraphics{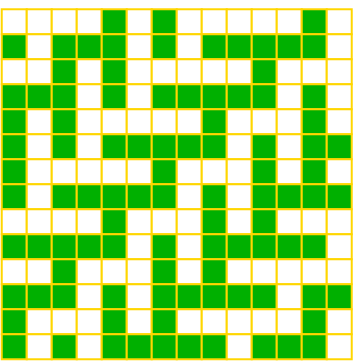}

17477-1* \hskip 65pt 17489-1* \hskip 65pt 17669-1*
\caption{The 6 isonemal prefabrics of species 3 and order 20 that fall apart.}\label{fig12:}
\end{figure}
The appropriate results from each species are displayed.
Figures 12 (6 designs of species 3) and 13 (6 designs of species 6 of the second kind shown in Figure 11c) have the peculiarity that, lacking all rotational symmetry, they are an arbitrary choice of which way the design is displayed. 
The same prefabric could have its axes rotated from the displayed positions through any multiple of $90^\circ$ including upside down to produce four genuinely different pictures. 
The other prefabrics look different from the given diagram when rotated $90^\circ$ either way but the same upside down so that there are only two genuinely different pictures of each pattern in Figures 14--19.
\begin{figure}
\centering
\includegraphics{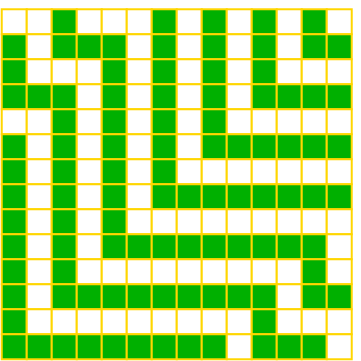}\hskip 5 pt \includegraphics{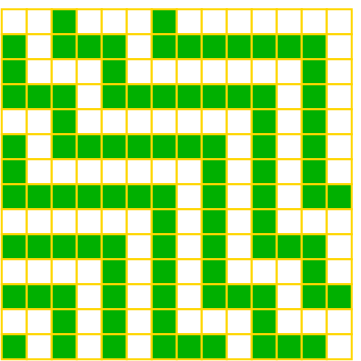}\hskip 5 pt \includegraphics{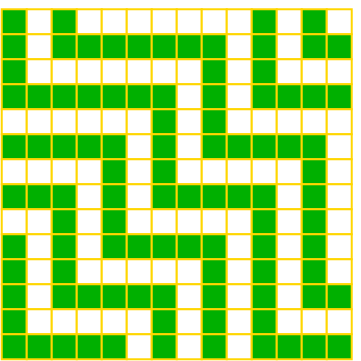}

1109-2* \hskip 65pt 4373-2* \hskip 65pt 5141-2*
\vskip 5pt
\includegraphics{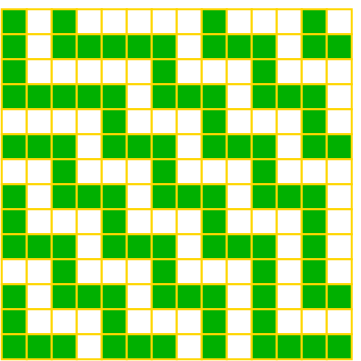}\hskip 5 pt \includegraphics{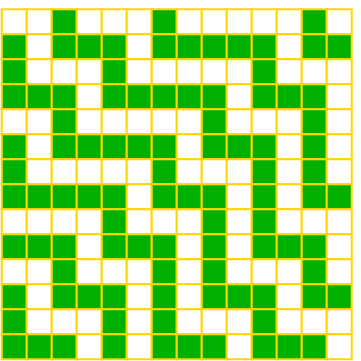}\hskip 5 pt \includegraphics{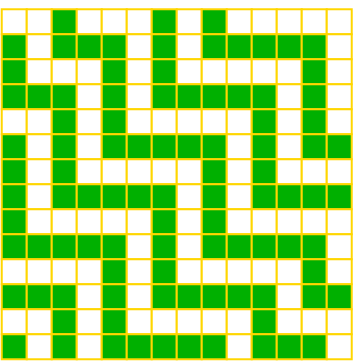}

\noindent 17477-2* \hskip 65pt 17489-2* \hskip 65pt 17669-2*
\caption{The 6 isonemal prefabrics of species 6 and order 20 that fall apart (mirrors of positive slope $5\delta$ apart).}\label{fig13:}
\end{figure}
\begin{figure}
\centering
\includegraphics{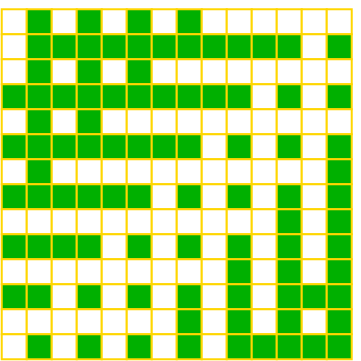}\hskip 5 pt \includegraphics{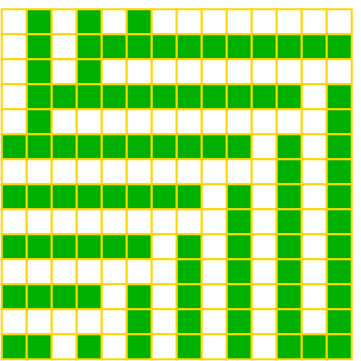}\hskip 5 pt \includegraphics{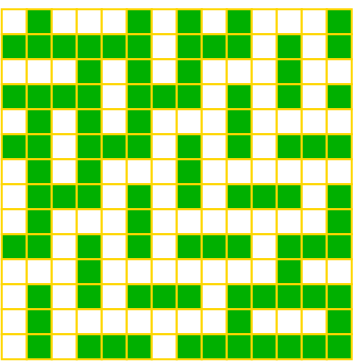}

341-1* \hskip 65pt 341-2* \hskip 65pt 4433-1*
\vskip 5pt
\includegraphics{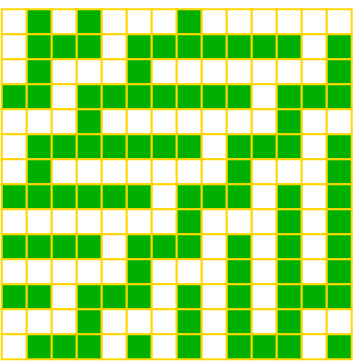}\hskip 5 pt \includegraphics{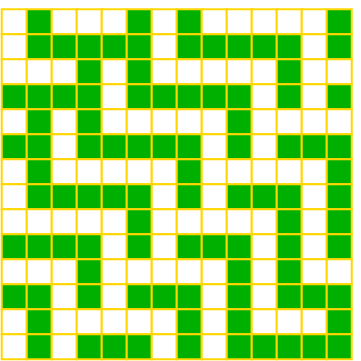}\hskip 5 pt \includegraphics{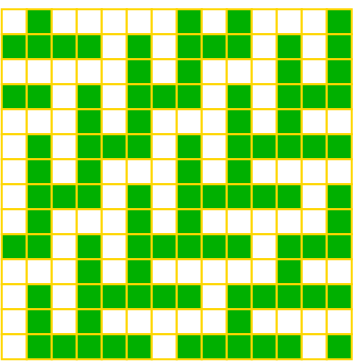}

4433-2* \hskip 65pt 16709-1* \hskip 65pt 16709-2*
\caption{The 6 isonemal prefabrics of species $15_o$ and order 20 that fall apart (side-preserving glide-reflection axes with negative slope $\delta$ apart).}\label{fig14:}
\end{figure}
\begin{figure}
\centering

\includegraphics{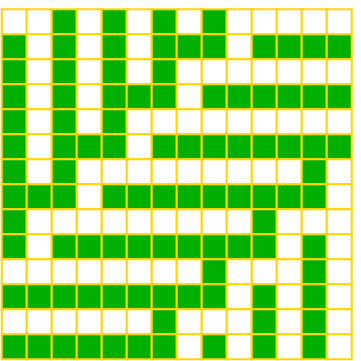}\hskip 5 pt \includegraphics{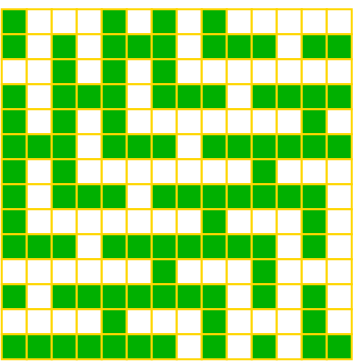}\hskip 5 pt \includegraphics{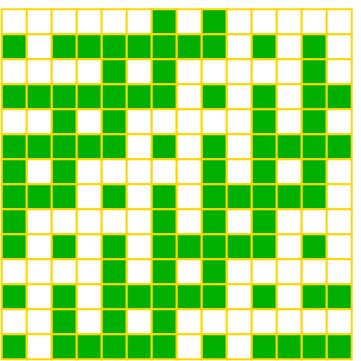}

1109-3* \hskip 65pt 4373-3* \hskip 65pt 5141-3*
\vskip 5pt
\includegraphics{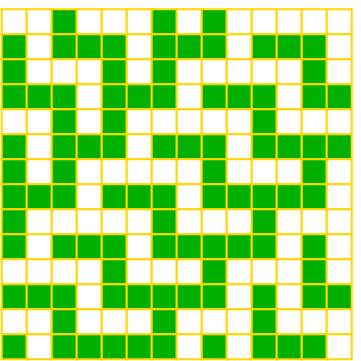}\hskip 5 pt \includegraphics{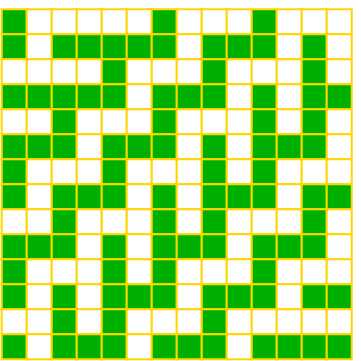}\hskip 5 pt \includegraphics{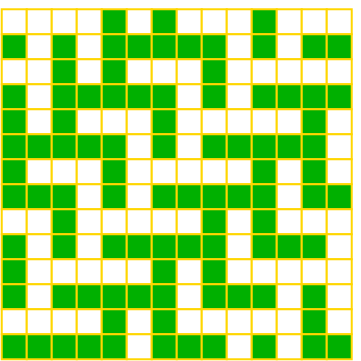}

17477-3* \hskip 65pt 17489-3* \hskip 65pt 17669-3*

\caption{The 6 isonemal prefabrics of species $15_o$ and order 20 that fall apart (side-preserving glide-reflection axes with positive slope $5\delta$ apart).}\label{fig15:}
\end{figure}
\begin{figure}
\centering

\includegraphics{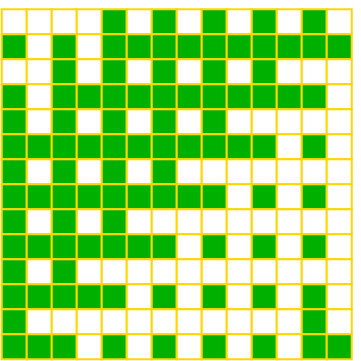}\hskip 5 pt \includegraphics{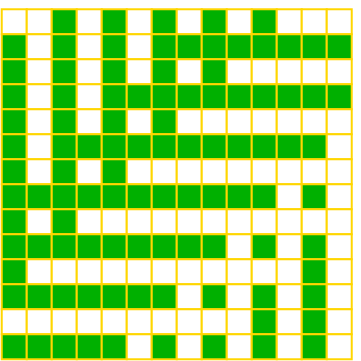}\hskip 5 pt \includegraphics{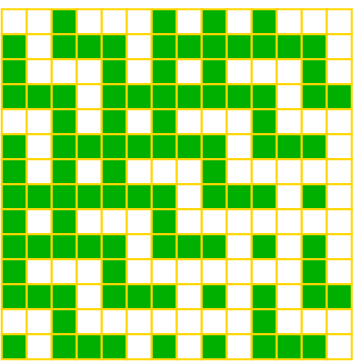}

341-3* \hskip 65pt 341-4* \hskip 65pt 4433-3*
\vskip 5pt
\includegraphics{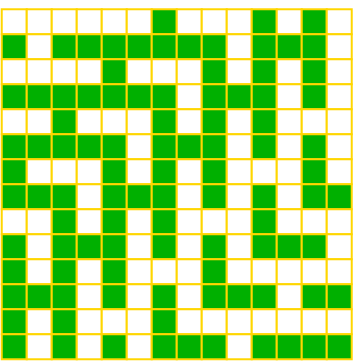}\hskip 5 pt \includegraphics{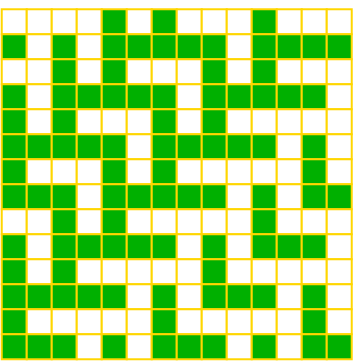}\hskip 5 pt \includegraphics{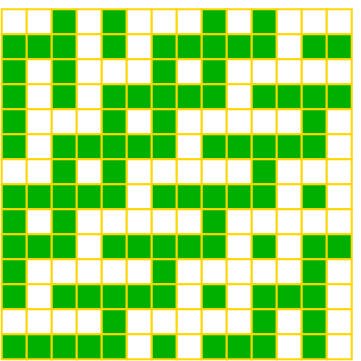}

4433-4* \hskip 65pt 16709-3* \hskip 65pt 16709-4*

\caption{The 6 isonemal prefabrics of species $19_o$ and order 20 that fall apart (mirrors with positive slope $5\delta$ apart).}\label{fig16:}
\end{figure}
\begin{figure}
\centering
\includegraphics{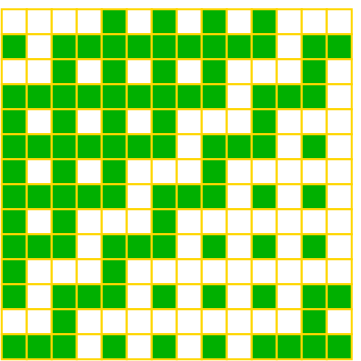}\hskip 5 pt \includegraphics{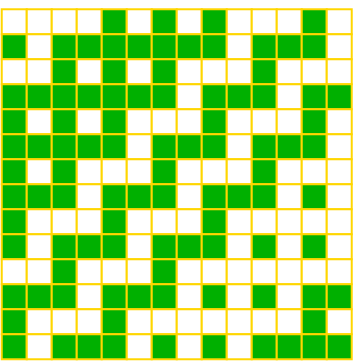}\hskip 5 pt \includegraphics{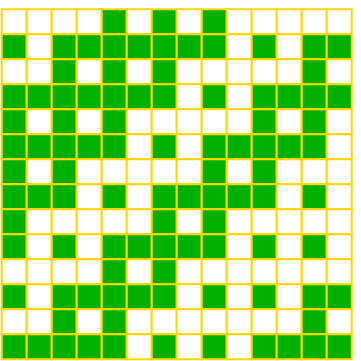}

1109-4* \hskip 65pt 4373-4* \hskip 65pt 5141-4*
\vskip 5pt
\includegraphics{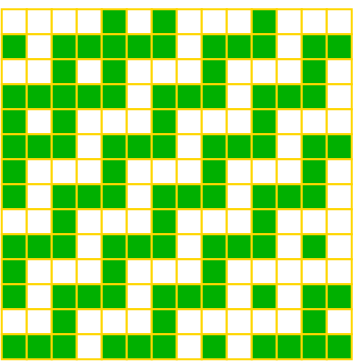}\hskip 5 pt \includegraphics{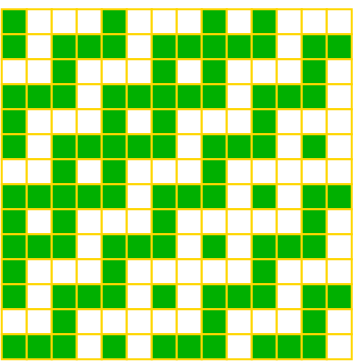}\hskip 5 pt \includegraphics{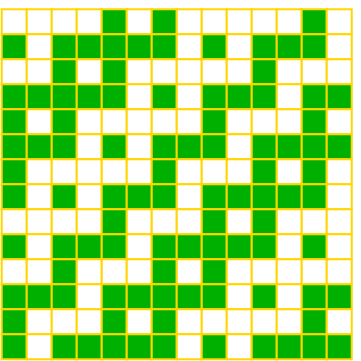}

17477-4* \hskip 65pt 17489-4* \hskip 65pt 17669-4*
\caption{The 6 isonemal prefabrics of species $19_o$ and order 20 that fall apart (mirrors with negative slope $\delta$ apart).}\label{fig17:}
\end{figure}
\begin{figure}
\centering

\includegraphics{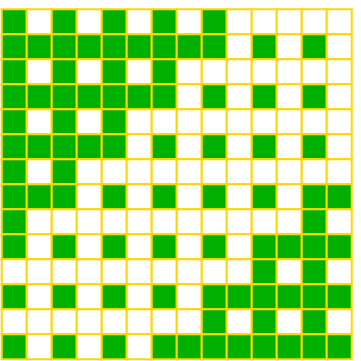}\hskip 5 pt \includegraphics{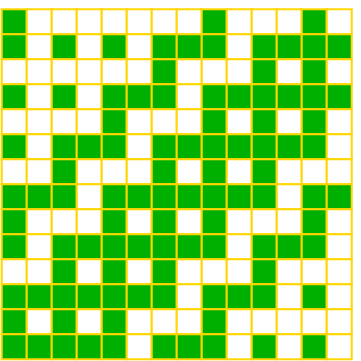}\hskip 5 pt \includegraphics{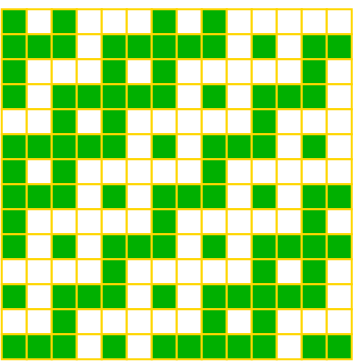}\\

341-5* \hskip 65pt 4433-5* \hskip 65pt 16709-5*
\caption{The 3 isonemal prefabrics of species $23_o$ and order 20 that fall apart (mirrors with negative slope $\delta$ apart).}\label{fig18:}
\end{figure}
\begin{figure}
\centering

\includegraphics{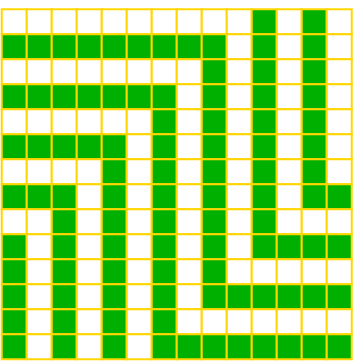}\hskip 5 pt \includegraphics{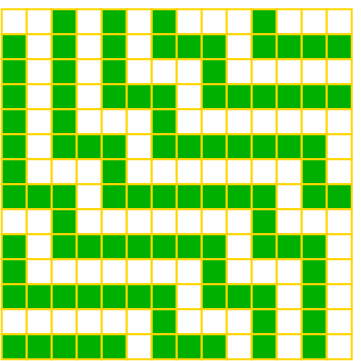}\hskip 5 pt \includegraphics{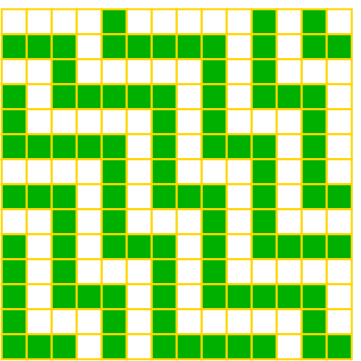}

341-6* \hskip 65 pt 4433-6* \hskip 65 pt 16709-6*

\caption{The 3 isonemal prefabrics of species $23_o$ and order 20 that fall apart (mirrors with positive slope $5\delta$ apart).}\label{fig19:}
\end{figure}

Collectively Figures 12--19 extend the catalogue of isonemal prefabrics that fall apart in \cite{JA} to order 20 for those whose alternate warps and wefts lift off (as though thinly striped, even genus) but not those where warps and wefts lift off in adjacent pairs (as though thickly striped, genus V).

The 42 patterns/designs are based on only 9 sequences.
The 3 palindromes appear 6 times each, and the 6 non-palindromes 4 times.
\section{Thin striping and falling apart}
\noindent Can designs of isonemal prefabrics that fall apart always be produced by thinly striping isonemal fabrics?
This question was answered in the affirmative in \cite{JA}, but the answer depended on an algorithm for producing a fabric for colouring by thin striping that does not always work.
If one wants to go from the design of a prefabric that falls apart to the design of a fabric that would colour by striping to look like it, one will obviously reverse the colours of the irredundant cells in the predominantly dark rows and not those in the predominantly pale rows precisely so that the striping would give the starting design.
The only question really is what to do with the redundant cells, and the proposed algorithm made a wrong choice of making them all pale.
Since a number of symmetry groups require side-preserving---hence colour-reversing---glide reflections and half-turns, that cannot be done consistently, and the right thing to do is to colour them orbit by orbit in accordance with the requirements of the relevant symmetry group.
This works in lots of examples, helped by the fact that side-preserving glide-reflections with odd glides and axes through redundant cells or even glides and axes through irredundant cells and side-reversing glide-reflections with odd glides and axes through irredundant cells or even glides and axes through redundant cells in the given design turn into side-preserving glide-reflections of what were irredundant cells on account of the colour reversals in the predominantly dark rows.
But there seems to be no reason to think that every design so produced needs to be isonemal.

A way in which it can fail to be isonemal is illustrated in Figure 20 by an isonemal prefabric of species 6 that falls apart (a) and a specification (b) of the irredundant cells of a partial design for a fabric that could produce the design (a) by thin striping.
\begin{figure}
\centering

\includegraphics{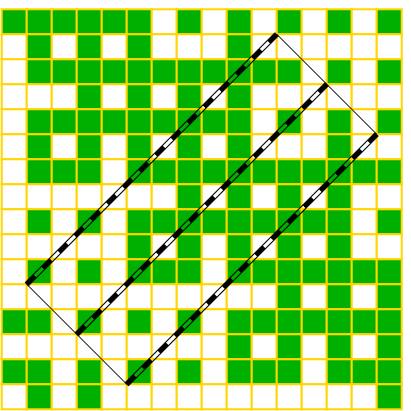}\hskip 10 pt \includegraphics{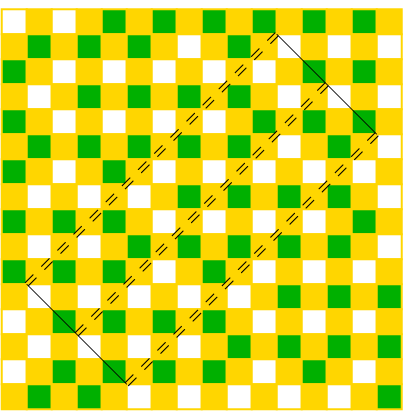}

(a) \hskip 110 pt (b)
\caption{a. Two $G_1$ lattice units = One $H_1$ lattice unit of an order-40 isonemal prefabric of species 6 that falls apart.\hskip 10 pt
b. Partial construction of a fabric that could be coloured to look like the design of (a). Dark and pale cells as required before striping, formerly redundant cells neutral.}\label{fig20:}
\end{figure}
The partial design has been produced by reversing the thin-striping algorithm for irredundant cells, i.e., the irredundant cells in the predominantly dark rows were complemented. 
The appearance of the partial construction of a fabric to colour, a pseudofabric because of the missing strands belonging in cells coloured neutral in Figure 20b, suggests that it should be a twillin with the same row-to-row offset 9 as the prefabric in Figure 20a but without the complementation that is a feature of the row-to-row translation that is a symmetry of the prefabric.
It does still have the glide-reflection of the prefabric although the complementing of alternate rows has destroyed the mirror symmetry.
While it looks like a simple matter to colour the formerly redundant cells in keeping with the glide-reflection symmetry, doing just that to a single cell reveals that any fabric thereby produced is not isonemal.
Consider the bottom cell on the central axis within the marked lattice units (dark in Figure 20a, neutral in Figure 20b).
Let its colour be $A$.
Five cells along that axis, the cell, again neutral in Figure 20b, must be $\overline{A}$ because the glide-reflection is side-preserving.
Four cells downward and to the right, the neutral cell must have colour $\overline{A}$ by the periodicity of both diagrams in that direction.
But that cell is the image of the initial cell with colour $A$ under the row-to-row translation with offset 9 that is the row-to-row symmetry of the formerly redundant cells of Figure 20b.
Such a row-to-row translation cannot be a symmetry of any completed construction of the fabric.
A glance at the lattice unit of the symmetry group of the irredundant cells plus the two cells coloured $A$ and $\overline{A}$ in it confirms by its dimensions, $4\delta\times10\delta$, that it is not a lattice unit of an isonemal fabric, since its parameters violate the isonemal constraint.
To be of species $1_m$, the width in $\delta$ units would have to be odd and relatively prime to 10.

In order to be consistent with the side-preserving glide-reflections of the irredundant cells with axes along the former axes, the redundant cells along those axes must be filled in with each set of five consecutive cells the colour complement of the previous set of five consecutive cells.
Other formerly redundant cells too need to be filled in in accordance with the group to make a design with group of Roth type $1_m$.
One observes that irredundant cells are invariant under translations of multiples of $5\delta$ parallel to the axes as well as glide-reflections with glide $5\delta$, but the redundant cells containing the axes are invariant under translations of only multiples of $10\delta$ along the axes because moving $5\delta$ they are complemented.
The period rectangle (lattice unit) must be $4\delta$ by $10\delta$, and so, no matter how the rest of the redundant cells are filled in, the period rectangle cannot be made smaller.
The species-6 design is isonemal because the lattice unit of $G_1$ is $4\delta$ by  $5\delta$, 4 and 5 being relatively prime.
The design partly produced cannot be isonemal because it is shaping up to be of species $1_m$ with lattice unit being $4\delta$ by $10\delta$. 
The standard isonemality constraint of \cite{P1} is violated.
(The translations under which it is invariant having those sizes in perpendicular directions, there is no transformation from strand to adjacent strand, as can be seen for the irredundant cells of, for example, the bottom two rows of Figure 20a.)

We must conclude that Theorem 2 of \cite{JA} is false and that the algorithm on which it is based does not always work either as stated there or improved here.
The improved algorithm works when it can work, but the counterexample of Figure 20 shows that no procedure can always work.


\ack
{Work on this material has been done at home and at Wolfson College, Oxford.
Richard Roth helped with the understanding of his papers.
Will Gibson made it possible for me to draw the diagrams with surprising ease from exclusively keyboard input. 
To them both and Wolfson College I make grateful acknowledgement.}


\end{document}